\documentclass[12pt]{amsart}
\usepackage{amsfonts,amssymb,amsmath}

\usepackage{mathrsfs}
\usepackage{bbm, upgreek}

\usepackage{tikz}
\usepackage{tikzit}
\usepackage{bm}

\tikzstyle{back}=[-, fill={rgb,255: red,193; green,215; blue,255}]
\tikzstyle{front}=[-, fill={rgb,255: red,135; green,157; blue,255}]
\tikzstyle{arrow}=[->]
\tikzstyle{wideArrow}=[width=100 pt, fill=none, ->]

\usepackage{xcolor}
\usepackage{comment}

\makeatletter
\@namedef{subjclassname@2020}{%
  \textup{2020} Mathematics Subject Classification}
\makeatother

\newcommand {\SB} {{\mathbb B}}
\newcommand {\SC} {{\mathbb C}}

\newcommand {\SR} {{\mathbb R}}
\newcommand {\SRdu} {{\SR^{d+1}}}
\renewcommand {\SS} {{\mathbb S}}
\newcommand {\ST} {{\mathbb T}}

\renewcommand {\phi} {{\varphi}}

\newcommand {\Dt} {{\Delta}}

\newcommand {\e} {{\varepsilon}}

\renewcommand{\O}{\Omega}

\newcommand{\Tu}{{T_\Omega}}
\newcommand{\bO}{{\partial\Omega}}
\newcommand{\tO}{{\widetilde\Omega}}

\newcommand {\la} {{\lambda}}

\newcommand {\bc} {{\mathfrak c}}

\newcommand{\be}{{\bf e}}

\newcommand {\tbn} {{\widetilde{\bn}}}

\newcommand {\bu} {{\bf u}}

\newcommand {\cD} {{\mathcal D}}

\newcommand {\cF} {{\mathcal F}}

\newcommand {\cO} {{\mathcal O}}

\newcommand {\cS} {{\mathcal S}}

\newcommand {\cT} {{\mathcal T}}

\newcommand {\cV} {{\mathcal V}}

\newcommand {\cW} {{\mathcal W}}

\newcommand {\tE} {{\widetilde E}}
\newcommand {\tF} {{\widetilde F}}

\newcommand {\tK} {{\widetilde K}}

\newcommand {\tR} {{\widetilde R}}

\newcommand {\tPhi} {{\widetilde{\raisebox{-0.2ex}[1.25ex][0ex]{$\Phi$}}}}

\newcommand {\hf} {{\widehat f}}

\newcommand {\hg} {{\widehat g}}

\def\Sym{\text{Sym}}

\def\SO{\text{SO}}

\newcommand {\diag} {\mbox{diag }}

\def\supp{\mathop{\rm supp}}

\newcommand {\hdot}[1]{\langle{#1}\rangle}

\def\lan#1#2{{\langle{#1},{#2}\rangle}}

\newcommand {\Scirc} {\raise.2ex\hbox{$\scriptstyle\circ$}}

\newcommand {\sae} {{\,a.e.\,}}
\newcommand {\mand} {{\quad\mbox{and}\quad}}

\renewcommand {\mid} {{\,\colon\,}}

\renewcommand {\mid} {{\,\,\,\colon\,\,\,}}

\newcommand{\Ba}[1]{\begin{array}{#1}}
\newcommand{\Ea}{\end{array}}
\newcommand{\Be}{\begin{equation}}
\newcommand{\Ee}{\end{equation}}
\newcommand{\Bea}{\begin{eqnarray}}
\newcommand{\Eea}{\end{eqnarray}}
\newcommand{\Beas}{\begin{eqnarray*}}
\newcommand{\Eeas}{\end{eqnarray*}}
\newcommand{\Benu}{\begin{enumerate}}
\newcommand{\Eenu}{\end{enumerate}}
\newcommand{\Bi}{\begin{itemize}}
\newcommand{\Ei}{\end{itemize}}

\newcommand{\BR}{\begin{Remark} \em}
\newcommand{\ER}{\end{Remark}}

\newcommand{\BE}{\begin{example} \em}
\newcommand{\EE}{\end{example}}

\newcounter{remark}

\newtheorem{theorem}[equation]{T{\hskip 0pt\footnotesize\bf HEOREM}}

\newtheorem{corollary}[equation]{C{\hskip 0pt\footnotesize\bf OROLLARY}}
\newtheorem{lemma}[equation]{L{\hskip 0pt\footnotesize\bf EMMA}}
\newtheorem{Remark}[equation]{R{\hskip 0pt\footnotesize\bf EMARK}}

\newtheorem{example}[equation]{E{\hskip 0pt\footnotesize\bf XAMPLE}}

\newcommand {\ProofEnd} {
             \begin{flushright} \vskip -0.2in $\Box$ \end{flushright}}

\def\bbone{{\mathbbm 1}}

\DeclareFontEncoding{FMS}{}{}
\DeclareFontSubstitution{FMS}{futm}{m}{n}
\DeclareFontEncoding{FMX}{}{}
\DeclareFontSubstitution{FMX}{futm}{m}{n}
\DeclareSymbolFont{fouriersymbols}{FMS}{futm}{m}{n}
\DeclareSymbolFont{fourierlargesymbols}{FMX}{futm}{m}{n}
\DeclareMathDelimiter{\VERT}{\mathord}{fouriersymbols}{152}{fourierlargesymbols}{147}

\def \<{\langle}
\def\>{\rangle}
\def \bn{{\mathbf n}}

\newcommand {\fts} {\footnotesize}
\newcommand {\scs} {\scriptsize}
\newcommand{\Dom}{{\mbox{\rm Dom }}}

\begin{document}

\title[Local cone multipliers and Cauchy-Szeg\"o projections] {Local cone multipliers and Cauchy-Szeg\"o projections
in bounded symmetric domains}

\author{Fernando Ballesta Yagüe}
\address{Fernando Ballesta Yagüe 
\\
Departamento de An\'alisis Matem\'atico
\\
Facultad de Ciencias Matem\'aticas
\\
Universidad Complutense de Madrid
\\
28040 Madrid, Spain} \email{ferballe@ucm.es}

\author{Gustavo Garrig\'os}
\address{Gustavo Garrig\'os
\\
Departamento de Matem\'aticas
\\
Universidad de Murcia
\\
30100 Murcia, Spain} \email{gustavo.garrigos@um.es}

\thanks{Both authors partially supported by grant 21955-PI22 from Fundaci\'on S\'eneca (Regi\'on de Murcia, Spain). F.B.Y. also supported by grant PID2020-113048GB-I00 funded by
MCIN/AEI/10.13039/501100011033, and Grupo UCM-970966 (Spain), and benefited from an FPU Grant FPU21/06111 from Ministerio de Universidades (Spain). G.G. also supported by grant PID2022-142202NB-I00 
from Agencia Estatal de Investigaci\'on (Spain).}



\begin{abstract}
We show that the cone multiplier satisfies local $L^p$-$L^q$ bounds only in the trivial range
$1\leq q\leq 2\leq p\leq\infty$. To do so, we suitably adapt to this setting the proof of Fefferman for the ball multiplier.
As a consequence we answer negatively a question by B\'ekoll\'e and Bonami \cite{BB}, regarding the continuity from $L^p\to L^q$ of the Cauchy-Szeg\"o projections
associated with a class of bounded symmetric domains in $\SC^n$ with rank $r\geq2$.
\end{abstract}

\maketitle

\section{Introduction}
\setcounter{equation}{0}\setcounter{footnote}{0}
\setcounter{figure}{0}

Let $d\geq2$. We denote the (forward) light-cone in $\SR^{d+1}$ by
\Be\label{O}
\O=\Big\{\xi=(\xi_1,\xi')\in\SR\times\SR^d\mid \xi_1>|\xi'|\Big\}.
\Ee
The cone multiplier is the operator $f\mapsto Sf$, defined  
by
\Be\label{S}
Sf(x):=\int_\O \hf(\xi)\,e^{2\pi i \hdot{x,\xi}}\,d\xi, \quad x\in\SR^{d+1},
\Ee
say for $f\in \cS(\SR^{d+1})$,
 where $\hf$ denotes the usual Fourier transform
\[
\hf(\xi)=\int_{\SR^{d+1}} f(y)\,e^{-2\pi i \hdot{y,\xi}}\,dy, \quad \xi\in\SR^{d+1}.
\]
We are interested in the problem of \emph{local $(p,q)$-boundedness} of $S$, meaning that for every
bounded set $B\subset\SR^{d+1}$, the operators 
\Be\label{SB}
S^B(f):=\bbone_B\, S\big(f\,\bbone_B)
\Ee
map $L^p(\SRdu)\to L^q(\SRdu)$ (with constants possibly depending on $B$). Clearly, it suffices to restrict $B$ to the family of balls $\SB_r=\{x\in\SR^{d+1}\mid |x|< r\}$ with $r>0$;
moreover, by a homogeneity argument, the problem is actually equivalent to consider one single ball $\SB=\SB_{r_0}$, for some fixed $r_0>0$, that is,
we wish to determine the values of $1\leq p,q\leq \infty$ so that, for some $C>0$ one has
\Be
\big\|Sf\big\|_{L^q(\SB)}
\leq \,C\,\|f\|_{L^p(\SB)}, \quad \mbox{$\forall\;f\in C^\infty_c(\SB)$.}
\label{Lpq}
\Ee
This local estimate holds trivially when $1\leq q\leq 2\leq p\leq\infty$, as a consequence of the boundedness of $S$ in $L^2(\SR^{d+1})$ and 
H\"older's inequality in the ball $\SB$. Our goal is to show that \eqref{Lpq} can only hold in this range;
see Figure \ref{fig1}.
The motivation for this question comes from several complex variables, and is discussed at the end of this section. The best range previously known for this problem, due to B\'ekoll\'e and Bonami \cite{BB}, is shown on the left of Figure \ref{fig1}.

\begin{theorem}\label{th1}
Let $1\leq p,q\leq \infty$. Then \eqref{Lpq} holds if and only  if
\[
1\leq q\leq 2\leq p\leq\infty.
\]
\end{theorem}

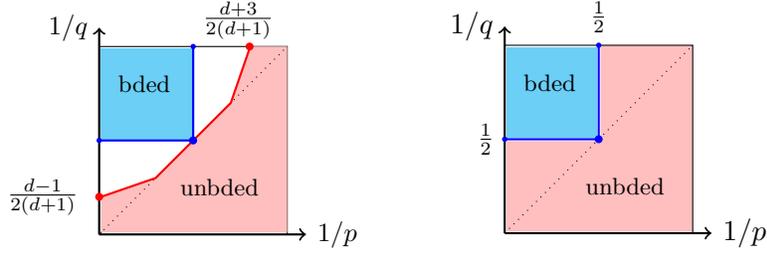
\begin{figure}
       \begin{tikzpicture}[scale=.5]

   \draw [->, thick]       (0,0)-- (5.5,0) node (xaxis) [right] {\fts $1/p$};
    \draw[->,thick] (0,0)--(0,5.5) node (yaxis)[left] { \fts $1/q$};
    \draw(0,0)--(5,5);
    \draw(0,5)--(5,5)--(5,0);
   \fill[cyan!50!white]
(2.5,2.5)--(2.5,4.95)--(0.05,4.95)--(0.05,2.5)--cycle;

\draw[thick,blue] (0,2.5)--(2.5,2.5)--(2.5,5);
\fill [blue] (2.5,2.5) circle (3pt);
\draw (1.2, 4) node {{\scs bded}};

\fill[pink] (0.05,0.05)--(0.05,1)--(1.5,1.5)--(3.5,3.5)--(4,5)--(5,5)--(5,0.05)--cycle;

\draw[dotted] (0,0)--(5,5);

\draw (3.7,5.7) node
{\fts $\tfrac{d+3}{2(d+1)}$};

\draw (-1.5,1) node
{\fts $\tfrac{d-1}{2(d+1)}$};

\draw[thick,red] (2.5,2.5)--(3.5,3.5)--(4,4.95);
\draw[thick,red] (2.5,2.5)--(1.5,1.5)--(0.05,1);

\fill [red] (4,5) circle (3pt); 
\fill [red] (0,1) circle (3pt); 
\fill [blue] (2.5,2.5) circle (3pt);
\fill [blue] (0,2.5) circle (2pt); 
\fill [blue] (2.5,5) circle (2pt);


\draw (3.2, 1.25) node {{\scs unbded}};

\end{tikzpicture} \quad\quad
\begin{tikzpicture}[scale=.5]

   \draw [->, thick] (0,0)-- (5.5,0) node (xaxis) [right] {\small$1/p$};
    \draw[->,thick] (0,0)--(0,5.5) node (yaxis)[left] {\small$1/q$};

\fill[pink] (0.05,0.05)--(0.05,2.45)--(2.55,2.45)--(2.55,5)--(5,5)--(5,0.05)--cycle;
    
    \draw[dotted](0,0)--(5,5);
    \draw(0,5)--(5,5)--(5,0);
		
		\draw (-0.5,2.5) node {\footnotesize$\tfrac12$};
		\draw (2.5,5.75) node {\footnotesize$\tfrac12$};

     \fill[cyan!50!white]
(2.45,2.55)--(2.45,4.95)--(0.05,4.95)--(0.05,2.55)--cycle;
		
		\draw[thick, blue] (0,2.5)--(2.5,2.5)--(2.5,5);
	\fill [blue] (2.5,2.5) circle (3pt); 
	\fill [blue] (0,2.5) circle (2pt); 
	\fill [blue] (2.5,5) circle (2pt);

\draw (1.2, 4) node {{\scs bded}};

\draw (3.2, 1.25) node {{\scs unbded}};

\end{tikzpicture}
\caption{Region of local $(p,q)$-boundedness for $S$: previous results from \cite{BB} on the left figure; the range in Theorem \ref{th1} on the right figure. }
\label{fig1}
\end{figure}

\

We remark that this question only concerns the upper triangle in Figure \ref{fig1}, since the case $p<q$ can be easily disregarded by homogeneity. Indeed, given a fixed $g\in C^\infty_c(\SR^{d+1})$, $g\not\equiv0$, if $R$ is large enough so that
$\supp g\subset R\,\SB$, then  \eqref{Lpq} applied to $f=g(R\cdot)$, together with a change of variables, implies that
\[
\|Sg\|_{L^q(\SB_R)}\leq \, C\,R^{\frac {d+1}q-\frac {d+1}p}\,\|g\|_{L^p(\SR^{d+1})},
\]
which letting $R\to\infty$ gives a contradiction.
If $p=q$, this same argument shows that \eqref{Lpq}
is equivalent to $S$ being globally bounded in $L^p(\SR^{d+1})$, which is false for all $p\not=2$ 
by the theorems of Fefferman and de Leeuw \cite{deL,Fef}. So \eqref{Lpq} becomes relevant only
when $p>q$.

\

A motivation for this question comes from a problem in several complex variables proposed by B\'ekoll\'e and Bonami in \cite{BB}, that we describe next. 
Consider the following bounded symmetric domain of $\SC^{n}$, $n\geq3$,
\Be\label{lie}
\cD=\Big\{z\in\SC^n\mid 2|z|^2-1<\big|\sum_{j=1}^nz_j^2\big|^2<1\Big\}.
\Ee
This domain, introduced by \'E. Cartan in  \cite[p.149]{car}, is usually called the \emph{Lie ball} of $\SC^n$, and is conformally equivalent (via a Cayley transform) to the tube domain 
\[
\cT=\SR^n+i\O
\]
over the cone $\O$ in \eqref{O} (with $n=d+1$); see e.g. \cite[p. 263]{vagi}. 

In operator theory, many structural properties on a bounded symmetric domain $D$ of $\SC^n$ (in particular, part of the theory of Hardy spaces $H^p$) depend crucially on the $L^p$-continuity of the associated \emph{Cauchy-Szeg\" o projection} $\SS_D$, see e.g.
\cite{zhu, zhu2}. When $D$ is the unit ball of $\SC^n$ (or a domain of rank 1), it is a classical fact that $\SS_D$ is bounded in $L^p$ for all $1<p<\infty$, see \cite{KV71}. However, for the Lie ball $\cD$ (whose rank is 2), B\'ekoll\'e and Bonami proved that the Cauchy-Szeg\"o projection $\SS_\cD$ maps $L^p\to L^p$ only if $p=2$, see \cite[Theorem 6]{BB}. 

It was asked in \cite{BB} whether $\SS_\cD$ could be bounded from $L^p\to L^q$ when $p\not=q$. Moreover, 
a transference principle given in \cite{BB} shows that this property is actually equivalent to the local $(p,q)$-boundedness of the cone multiplier $S$ (with $n=d+1$), that is to the validity of the inequality \eqref{Lpq} that we discussed before. This equivalence was used in \cite[Theorem 7]{BB} to provide a counterexample for the case $(L^\infty, L^q)$ if $q\geq 2n/(n-2)$, which by duality also applies to the case $(L^p, L^1)$ with $p\leq 2n/(n+2)$; see left of Figure \ref{fig1}. The question, however, was left open in other cases.

Here
we shall answer this question completely, not only for the Lie balls, but also for general symmetric domains of tube type.

\begin{corollary}\label{cor1}
Let $D\subset \SC^n$ be an irreducible bounded symmetric domain of tube type with rank $r\geq2$. If $1\leq p,q\leq \infty$, 
then $\SS_D$ extends continuously from $L^p$ into $L^q$
if and only  if
\[
1\leq q\leq 2\leq p\leq\infty.
\]
\end{corollary}

We remark that, quite recently, a similar result about failure of $L^p\to L^q$ continuity, in the different context of (Riesz) projections in the infinite torus $\ST^\infty$, has been obtained by Konyagin et al., see \cite[Corollary 2.2]{KQSS22}.

\

The structure of the paper is as follows. In \S\ref{S_2} we present a detailed proof of Theorem \ref{th1}, by adapting the original argument of Fefferman. In \S\ref{S_3} we prove Theorem \ref{th1b}, which is an extension of Theorem \ref{th1} to general symmetric cones $\O$ of rank $r\geq 2$; this part makes use of the Jordan algebra structure to reduce matters to the light-cone setting. We devote \S\ref{S_4} to present the concepts related with
Cauchy-Szeg\"o projections, including a sketch of the proof of the transfer principle; see Theorem \ref{th_transf}. Finally, in \S\ref{S_dual}, we give a small application of Corollary \ref{cor1} to the duality of Hardy spaces.


\section{Proof of Theorem \ref{th1}}\label{S_2}
\setcounter{equation}{0}\setcounter{footnote}{0}
\setcounter{figure}{0}

Our approach to Theorem \ref{th1} will follow closely the original proof of C.\! Fefferman for the ball multiplier \cite{Fef} (see also \cite[Chapter 10]{stein}), with the Besicovitch set construction given in \cite{Fef} now adapted to fit the geometry of the cone. 
We remark that the standard procedure to disprove the boundedness of $S$ in $L^p(\SR^{d+1})$ (for $p\not=2$), as a direct consequence of Fefferman's result for the $d$-ball and DeLeeuw's theorem (see e.g., \cite[Example 2.5.17]{Gra1}), does not seem so straightforward to carry out in the local setting that we are interested in. 

\subsection{Proof of Theorem \ref{th1}}\label{S_21}

By duality, it suffices to consider the case $1\leq p<2$, and assume for contradiction that
\[
M:=\big\|S\big\|_{L^p(\SB)\to L^1(\SB)}<\infty.
\]
Extending to $\ell^2$-valued operators (or using Khintchine's inequalities, see \cite{MZ39} or \cite[Theorem 5.5.1]{Gra1}), this implies that
\Be
\Big\|\big(\sum_j|Sf_j|^2\big)^\frac12\Big\|_{L^1(\SB)}\,\leq\, c_p\,M\,\Big\|\big(\sum_j|f_j|^2\big)^\frac12\Big\|_{L^p(\SB)},
\label{v1}
\Ee
for some constant $c_p>0$, and every finite collection $\{f_j\}\subset L^p_c(\SB)$.

Our first step is to prove a Meyer type lemma for the operator $S$ (in analogy to \cite[Lemma 1]{Fef}). Namely, given a finite collection of light-ray vectors $\bn_j=(1,\bu_j)\in\bO$, with $|\bu_j|=1$, we let $\tbn_j=(-1,\bu_j)$ and
 define the half-spaces
\begin{equation}\label{Pi}
\Pi_j:=\Big\{\xi\in\SR^{d+1}\mid \lan\xi{\tbn_j}<0\Big\}    
\end{equation}
and the corresponding multiplier operators
\[
H_jf:=\cF^{-1}\big(\bbone_{\Pi_j}\,\hf\big).
\]
\begin{lemma}\label{L_meyer}
With the above notation, it holds
\Be
\Big\|\big(\sum_j|H_jf_j|^2\big)^\frac12\Big\|_{L^1(\SB)}\,\leq\, c_p\,M\,\Big\|\big(\sum_j|f_j|^2\big)^\frac12\Big\|_{L^p(\SB)},
\label{v2}
\Ee
for every $\{f_j\}\subset L^p_c(\SB)$.
\end{lemma}
\begin{proof}
By density, it suffices to consider functions $\{f_j\}\subset C^\infty_c(\SB)$. Given $R\geq1$, let $g_j(x)=e^{-2\pi iR\lan x{\bn_j}}\,S\big(e^{2\pi iR\lan{\cdot}{\bn_j}}f_j\big)(x)$, so that
\[
\hg_j=\bbone_{\O}(\cdot+R\bn_j)\,\widehat{f_j}=\bbone_{\O-R\bn_j}\,\widehat{f_j}
\]
We now observe that 
\Be\label{cupO}
\bigcup_{R\geq1}(\O-R\bn_j)   =  \big\{\xi\in\SRdu\,\colon\, \langle\xi,
\tbn_j\rangle<0\big\}=\Pi_j,
\Ee
see Figure \ref{fig2}. Thus, by dominated convergence we have
\[
\lim_{R\to\infty}|g_j|=
\big|\cF^{-1}\big[\bbone_{\Pi_j}\,\hf_j\big]\big|=\big|H_j(f_j)\big|.
\]
Hence, Fatou's lemma  and \eqref{v1} give
\Beas 
\Big\|\big(\sum_j|H_j(f_j)|^2\big)^\frac12\Big\|_{L^1(\SB)} & \leq & 
\liminf_{R\to \infty}\Big\|\big(\sum_j|g_j|^2\big)^\frac12\Big\|_{L^1(\SB)}\\
& \leq & c_p\,M\,\Big\|\big(\sum_j|f_j|^2\big)^\frac12\Big\|_{L^p(\SB)}.
\Eeas
\vskip-20pt
\end{proof}

\begin{figure}[ht]
    \centering
    \begin{tikzpicture}[scale=0.4]
	\begin{pgfonlayer}{nodelayer}
		\node [style=none, label={below: $-R\bm{n}_j$}] (2) at (-7, -7) {};
		\node [style=none] (4) at (5, 5) {};
		\node [style=none] (5) at (-19, 5) {};
		\node [style=none] (0) at (0, 0) {};
		\node [style=none] (1) at (0, 5) {};
		\node [style=none] (3) at (-5, 5) {};
		\node [style=none] (6) at (-7, 5) {};
		\node [style=none, label={right:$\xi_1$}] (7) at (0, 9) {};
		\node [style=none, label={above:$\xi_2$}] (8) at (7, 0) {};
		\node [style=none, label={above:$\xi_3$}] (9) at (5, 2) {};
		\node [style=none, label={right: $\bm{n}_j$}] (10) at (2, 2) {};
        \node [style=none, label={right: $\tbn_j$}] (15) at (2, -2) {};
		\node [style=none, label={above: $\Omega-R\bm{n}_j$}] (11) at (-7, 7.5) {};
		\node [style=none, label={above: $\Omega$}] (12) at (6, 5) {};
	\end{pgfonlayer}
	\begin{pgfonlayer}{edgelayer}
		\draw [style=back] (4.center)
			 to (2.center)
			 to (5.center)
			 to [bend left=90, looseness=0.30] cycle;
		\draw [style=back] (0.center)
			 to (3.center)
			 to [bend left=90, looseness=0.25] (4.center)
			 to cycle;
		\draw [style=front] (2.center)
			 to (5.center)
			 to [bend right, looseness=0.50] (4.center)
			 to cycle;
		\draw [bend right=60, looseness=0.25] (3.center) to (4.center);
		\draw [style=front] (3.center)
			 to [bend right=60, looseness=0.25] (4.center)
			 to (0.center)
			 to cycle;
		\draw [bend right, looseness=0.50] (5.center) to (4.center);
		\draw [style=arrow] (0.center) to (8.center);
		\draw [style=arrow] (0.center) to (9.center);
		\draw [style=arrow] (0.center) to (7.center);
		\draw [style=arrow, very thick] (0.center) to (10.center);
        \draw [style=arrow, very thick] (0.center) to (15.center);
		\draw [style=arrow] (0.center) to (2.center);
	\end{pgfonlayer}
\end{tikzpicture}
    \caption{Translated cones $\Omega-R\bn_{j}$, filling the half-space $H_j$.}
     \label{fig2}
\end{figure}
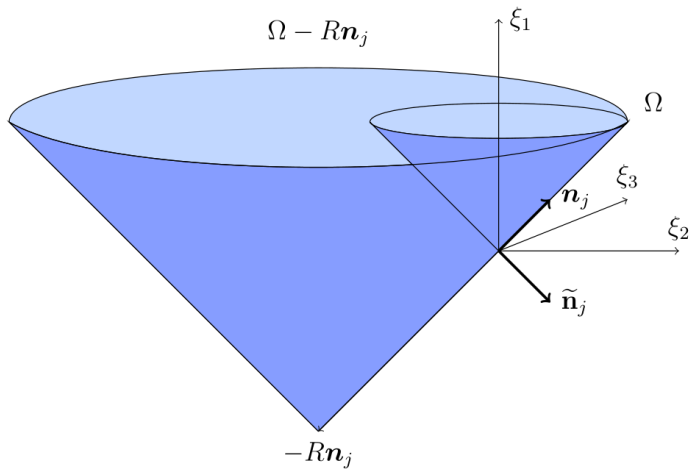


\

By taking a specific choice of directions $\{\bn_j\}_{j=1}^{N}$ and functions $\{f_j\}_{j=1}^{N}$ in Lemma \ref{L_meyer}, we shall reach a contradiction.
In the same spirit of Fefferman's construction, we will choose $f_j=\bbone_{F_j}$ for a suitable collection of sets $\{F_j\}_{j=1}^{N}$ with a great overlap, so that the right-hand-side of \eqref{v2} is small. At the same time, we ensure that the left-hand side of \eqref{v2} is greater than a constant, by selecting the sets $F_j$ as boxes made up of $\frac1{\sqrt2}$-segments in the direction $\tbn_j$ and so that the translates $\tF_j=F_j+5\tbn_j$ are pairwise disjoint, which will imply that $\sum_j|H_j(\bbone_{F_j})|^2\gtrsim \sum_j \bbone_{\tF_j}$ has a large support.

\

For clarity, we first give the construction in $\mathbb{R}^3$ (i.e., $d=2$), and later explain how to derive the case $d>2$. Given $\varepsilon>0$, we start with a collection $\{R_j\}_{j=1}^N$ of $1\times \frac{1}{N}$ rectangles in $\SR^2$ with the properties
\Be
\Big|\bigcup_{j=1}^N R_j\Big|<\varepsilon,
\Ee
 and if $\bu_j\in\SR^2$ denotes the long axis direction of $R_j$, then the translated rectangles
 \Be
 \tR_j:=R_j+5\bu_j, \quad j=1,\ldots, N 
\Ee
are pairwise disjoint, so in particular
\Be
\Big|\biguplus_{j=1}^N \tR_j\Big|=1.
\Ee
These rectangles can be obtained with the usual Besicovitch construction, as e.g. in \cite[Theorem X.1]{stein}. We may also assume that $R_j\cup \tR_j$ are all contained a fixed ball, say  
$\{x\in\SR^2\mid |x|\leq 10\}$, for every $j$ and $N$.

We next define in $\SR^3$ the vectors $\tbn_j=(-1,\bu_j)$ and the boxes 
\[
E_j:=[0,1]\times R_j\mand \tE_j:=E_j+5\tbn_j,
\]
which inherit the properties
\[
\big|\bigcup_{j=1}^N E_j\big|<\e \mand \big|\biguplus_{j=1}^N \tE_j\big|=\sum_{j=1}^N|\tE_j|=1.
\]
Indeed, the translated boxes $\tE_j$ are pairwise disjoint because their projections over the plane $\{0\}\times\mathbb{R}^2$ are the rectangles $\Tilde{R}_j$, which are disjoint; see Figure \ref{fig3}. 

Finally, we select a rotated box within $E_j$ made up of parallel $\tfrac1{\sqrt2}$-segments in the direction of $\tbn_j$, as in Figure \ref{fig3}; namely, if $\bc_j\in\SR^2$ is the center of $R_j$ and
\[
E_j=(0,\bc_j)+\Big\{
\lambda_1 \be_1+ \la_2 (0,\bu_j)
+
\lambda_3 (0,\bu_j^{\perp})\mid
\lambda_1\in[0,1],\;
|\lambda_2|\leq\tfrac12,\; 
|\la_3|\leq \tfrac{1}{2N}
\Big\},
\]
then we let
\[
F_j=(0,\bc_j)+\Big\{
\lambda_1 (1,-\bu_j)+ \la_2 (1,\bu_j)
+
\lambda_3 (0,\bu_j^{\perp})\mid
\lambda_1, \la_2\in[0,\tfrac12],\; 
|\la_3|\leq \tfrac{1}{2N}
\Big\}.
\]
These sets $F_j\subset E_j$ satisfy
\begin{itemize}
    \item $F_j$ are boxes, because $(1,-\bu_j)$, $(1,\bu_j)$ and $(0,\bu_j^{\perp})$ are orthogonal vectors
    \item $F_j$ has one side in the direction $\tbn_j=(-1,\bu_j)$, the normal direction of $\Pi_j$
    \item The boxes $F_j$ have a great overlap, since
    \[
    \left|\bigcup_{j=1}^N F_j\right|
    \leq 
    \left|
    \bigcup_{j=1}^N E_j
    \right|
    <\varepsilon
    \]
    \item The translated boxes $\widetilde{F}_j=F_j+5\tbn_j$ are pairwise disjoint (as they are subsets of $\widetilde{E}_j$, which are disjoint), and since $|F_j|=|E_j|/2=\frac{1}{2N}$ we have
    \[
    \big|\biguplus_{j=1}^N \tF_j\big|=\sum_{j=1}^N|\tF_j|=1/2.
    \]
\end{itemize}

\begin{figure}
\centering
\hspace{3cm}\includegraphics[width=12cm]{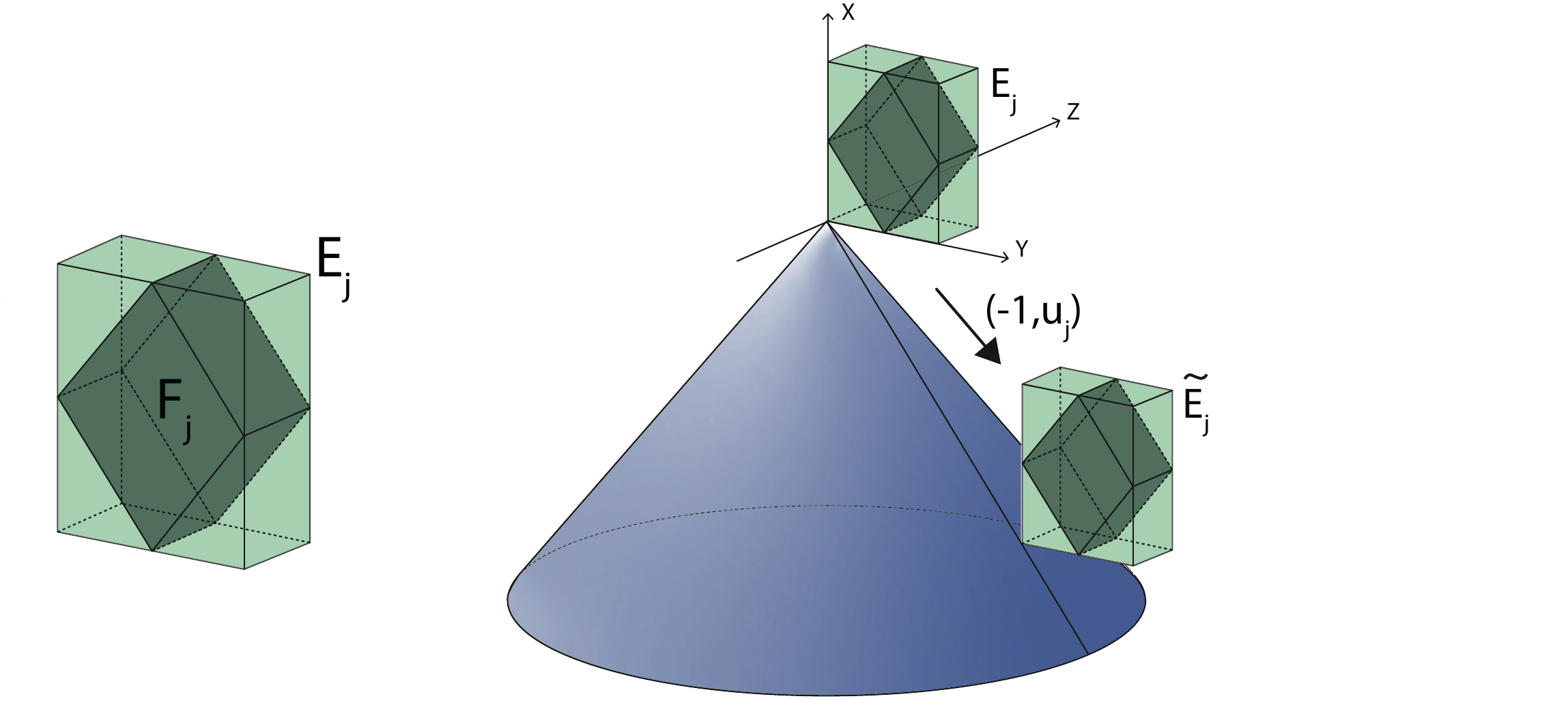}
\caption{Boxes $E_j$ and $F_j$ and their translations $\widetilde{E}_j$ and $\widetilde{F}_j$.}
\label{fig3}
\end{figure}

Using these sets we shall produce a contradiction with Lemma \ref{L_meyer}. We define  $f_j:= \bbone_{F_j}$, and let $\Pi_j$ be as in \eqref{Pi} (with our current choice of $\tbn_j=(-1,\bu_j)$). We assume that the ball $\SB=\SB_{r_0}$ is sufficiently large so that $F_j\cup\tF_j\subset\SB$ for all $j$ and $N$.

Now, the right-hand side of \eqref{v2}, by H\"older's inequality, can be bounded above by
\Bea
\Big\|\big(\sum_j|\bbone_{F_j}|^2\big)^\frac12\Big\|_{L^p(\SB)}
& \leq &  
\Big\|\big(\sum_j|\bbone_{F_j}|^2\big)^\frac12\Big\|_{L^2(\SB)}
\cdot
\left|\bigcup F_j\right|^{\frac{1}{p}-\frac{1}{2}}\nonumber\\
& = & \big(\sum_j|F_j|\big)^\frac12
\cdot
\left|\bigcup F_j\right|^{\frac{1}{p}-\frac{1}{2}}
<
\, \varepsilon^{\frac{1}{p}-\frac{1}{2}}. \label{eps}
\Eea

To estimate the left-hand side of \eqref{v2}, we first observe that,
by the construction of the sets $F_j$ we have
\[
\big|H_j \big(\bbone_{F_j}\big)\big|
\gtrsim
\bbone_{\widetilde{F}_j}.
\]
Indeed, after a translation and a rotation, this can be easily obtained from the fact that the 1-dimensional operator $Hf:=\cF^{-1}(\bbone_{(0,\infty)}\hf)$ satisfies
\[
\Big|H\big(\bbone_{[-\frac12,\frac12]}\big)(x)\Big| \,\geq\,\frac{c}{|x|}, \quad |x|\geq 1/2,
\]
see e.g. \cite[Chapter X, \S2.5.2]{stein}. 
Therefore, using the disjointness of the sets $\{\widetilde{F}_j\}$, we obtain
\Be
\Big\|\big(\sum_j|H_j(f_j)|^2\big)^\frac12\Big\|_{L^1(\SB)}
\gtrsim
\Big\|\big(\sum_j|\bbone_{\widetilde{F}_j} |^2\big)^\frac12\Big\|_{L^1(\SB)}
=
\left|\bigcup \widetilde{F}_j\right|
=
\frac{1}{2}.
\label{ctant}
\Ee
Thus, combining the estimates in \eqref{eps} and \eqref{ctant} with Lemma \ref{L_meyer}, and letting $\varepsilon\searrow0$, we reach the desired contradiction. This completes the proof when $d=2$.

\subsection{Proof of Theorem \ref{th1} when $d>2$}\label{S_d}
We now indicate how to produce a counterexample in $\mathbb{R}^{d+1}$, when $d>2$.  
We split the coordinates in $\SRdu=\SR^3\times \SR^{d-2}$, and apply Lemma \ref{L_meyer} to the functions
\Be\label{fjphi}
f_j
:=
\bbone_{F_j}\otimes \phi
\Ee
where $\phi\in C^\infty_c(\mathbb{R}^{d-2})$ is a fixed (non-zero) function supported in $\SB^{(d-2)}$, and $F_j$ are the sets constructed in the case of $\mathbb{R}^3$. We consider
the half-spaces $\Pi_j$ in \eqref{Pi} related to the light-ray vectors
\Be\label{nj}
\bn_j
:=
(1,\bu_j,0\ldots,0) \mand
\tbn_j
:=
(-1,\bu_j,0\ldots,0), 
\Ee
where $\bu_{j}\in\SR^2$ are the longest side directions of the Besicovitch rectangles $R_j$. In this setting it is clear that we have
\Be\label{Fjphi}
H_j(\bbone_{F_j}\otimes \phi)=H^{(3)}_j(\bbone_{F_j})\otimes \phi,
\Ee
where $H^{(3)}_j$ is the 3-dimensional multiplier operator from the previous setting. Thus, \eqref{v2} implies a corresponding 3-dimensional inequality
\[
\|\phi\|_1\,\Big\|\big(\sum_j|H_j^{(3)}(\bbone_{F_j})|^2\big)^\frac12\Big\|_{L^1(\SB^{(3)})}\,\leq\, c_p\,M\,\|\phi\|_p\,\Big\|\big(\sum_j|\bbone_{F_j}|^2\big)^\frac12\Big\|_{L^p(\SB^{(3)})},
\]
which we know it cannot hold. This completes the proof of Theorem \ref{th1} for all dimensions $d\geq2$.
\ProofEnd


\section{Extension of Theorem \ref{th1} to symmetric cones}
\label{S_3}
\setcounter{equation}{0}\setcounter{footnote}{0}
\setcounter{figure}{0}

In this section we show how to extend the validity of Theorem \ref{th1} from light-cones to general symmetric cones $\O$. This requires the use of appropriate coordinates in the underlying space $V=\SR^n$, regarding it as a Jordan algebra. In what follows we shall use standard notation and terminology from the text \cite{FK}. For readers less familiar with this terminology, one may consider as a guideline the example $V=\Sym(r,\SR)$ with $\O$ the cone of positive definite symmetric matrices. 

Let $\O$ be an irreducible symmetric cone in $V=\SR^n$. As discussed in \cite[Ch III]{FK}, this induces in $V$ a structure of Euclidean Jordan algebra, which we shall assume of rank $r\geq 2$. 
We denote by $\be$ the identity element in $V$, and by $\Dt$ (or $\det$) the associated determinant function. 

An idempotent is a vector $c\in V$ such that $c^2=c$; it is said to be a  \emph{primitive idempotent} if it is non-zero and cannot be written as the sum of two non-zero idempotents. In the case $V=\Sym(r,\SR)$ these are the matrices of the form $c_j=\diag(0,\ldots,1,\ldots,0)$ and their rotations $kc_jk'$, $k\in \SO(r)$.

A complete system of orthogonal primitive idempotents $\{c_1,\ldots, c_r\}$, with $c_1+\ldots+c_r=\be$, is called a \emph{Jordan frame}. Each Jordan frame induces in $V$ an orthogonal decomposition
\[
V=\bigoplus_{1\leq i\leq j\leq r} V_{ij},
\]
which formally allows to regard $V$ as
an algebra of symmetric matrices with entries in the subspaces $V_{ij}$. 
This is called a \emph{Peirce decomposition} of $V$, see \cite[Ch IV]{FK}. With this notation, the cone $\O$ can in particular  be described as 
\Be\label{minor}
\O=\Big\{\xi\in V\mid \Dt_\ell(\xi)>0, \quad \ell=1,2,\ldots, r\Big\},
\Ee
where the \emph{principal minor} $\Dt_\ell(\xi)$ is the determinant of the  projection of $\xi$ onto the Jordan subalgebra $V^{(\ell)}=\oplus_{1\leq i\leq j\leq \ell}V_{ij}$; see e.g. \cite[pp. 114, 121]{FK}. 

\begin{lemma}\label{L_c1}
Let $c_1$ be a primitive idempotent in $V$ and  $\bn:=\be-c_1$. Then
\[
\bigcup_{R\geq1}\big(\O-R\bn\big)=\Big\{\xi\in V\mid \lan\xi{c_1}>0\Big\}.
\]
\end{lemma}
\begin{proof}
For the inclusion ``$\subseteq$'', the orthogonality relation $\lan\bn{c_1}=0$ gives
\[
\lan{\xi-R\bn}{c_1}
=\lan\xi{c_1}>0, \quad \xi\in \O.
\]
For the converse inclusion ``$\supseteq$'', let $\xi\in V$ be such that $\lan\xi{c_1}>0$. 
We must show that, for some large $R\gg1$, it holds
\[
\xi+R\bn\in\O,
\]
or equivalently, using the principal minor characterization in \eqref{minor}, that
\Be\label{Dtl}
\Dt_\ell(\xi+R\bn)>0, \quad \ell=1,2,\ldots, r.
\Ee
When $\ell=1$, this is trivial since \[
\Dt_1(\xi+R\bn)=\lan{\xi+R\bn}{c_1}=
\lan\xi{c_1}>0.
\]Next, it is enough to prove \eqref{Dtl} for $\ell=r$, since in the other cases $\Dt_\ell(x)$ coincides with the determinant of the  projection of $x$ onto a new Jordan algebra $V^{(\ell)}$, so the same general argument will apply. 

Using the Peirce decomposition (see \cite[Ch IV]{FK}) with respect to $c_1$:
\[
V=V_{11}\oplus V_{\frac12} \oplus V_0,
\]
where $V_{\frac12}=V_{12}\oplus\ldots\oplus V_{1r}$, 
and $V_0$ is the subalgebra $\oplus_{2\leq i\leq j\leq r}V_{ij}$,
we write
\[
\xi+R\bn= \xi_{11} +\xi_{\frac12}+\big(\xi'+R\be'\big),
\]
where $\xi_{11}=\lan\xi{c_1}c_1$, and  $x'$ denotes the projection onto $V_0$ (note that $\bn=\be'$ is the unit in $V_0$).
Then, the properties of the determinant give
\Be\label{det2}
\Dt(\xi+R\bn)\,=\,\lan\xi{c_1}\,\cdot\,\Dt'\Big(\xi'+R\be' -\tfrac1{\lan\xi{c_1}}\big(\xi_{\frac12}^2\big)'\Big),
\Ee
where $\Dt'$ is the determinant in $V_0$; see \cite[p.139]{FK} or \cite[p.1751, line 2]{BGN14}. Now, in a Jordan algebra of rank $s$ we have
\[
\det(R\be+x)=R^s+\cO(R^{s-1})>0, \quad \mbox{if $R\geq R_0(x)$,}
\]
see e.g. \cite[pp. 28-29]{FK}.
Applying this result to the determinant $\Dt'$ in \eqref{det2}, we see that this expression is positive if $R$ is chosen sufficiently large. So, \eqref{Dtl} holds for a sufficiently large $R$, 
completing the proof of the inclusion ``$\supseteq$''.
\end{proof}

We shall now establish a Meyer type lemma for the multiplier operator 
\[
Sf:=\cF^{-1}\big[\bbone_\O\,\hf\big],
\]
where $\O$ is a general irreducible symmetric cone in $V$ (of rank $r\geq2$).
As in \S\ref{S_21}, for $1\leq p<2$, assume that a local inequality holds
\[
M:=\big\|S\big\|_{L^p(\SB)\to L^1(\SB)}<\infty.
\]
We pick a finite collection of vectors $\bn_j=\be- \bc_j$ in $\partial\O$, where each $\bc_j$ is a primitive idempotent in $V$, and define the collection of half-spaces
\Be\label{Pij2}
\Pi_j:=\Big\{\xi\in V\mid \lan\xi{\bc_j}>0\Big\},
\Ee
and the corresponding multiplier operators 
\Be\label{Hj2}
H_jf:=\cF^{-1}\big(\bbone_{\Pi_j}\,\hf\big).
\Ee
The next lemma is then a repetition of Lemma \ref{L_meyer}.

\begin{lemma}\label{L_meyer2}
With the above assumptions, it holds
\Be
\Big\|\big(\sum_j|H_jf_j|^2\big)^\frac12\Big\|_{L^1(\SB)}\,\leq\, c_p\,M\,\Big\|\big(\sum_j|f_j|^2\big)^\frac12\Big\|_{L^p(\SB)},
\label{v22}
\Ee
for every $\{f_j\}\subset L^p_c(\SB)$.
\end{lemma}
\begin{proof}
The same proof as in Lemma \ref{L_meyer} applies here, after replacing the identity in \eqref{cupO} by the new Lemma \ref{L_c1}.
\end{proof}

With this preparation, we can state the following extension of Theorem \ref{th1}.

\begin{theorem}
\label{th1b}
Let $\O\subset\SR^n$ be an irreducible symmetric cone of rank $r\geq2$. If $1\leq p,q\leq \infty$, then $\bbone_\O$ is a local $(p,q)$-Fourier multiplier
if and only  if
\[
1\leq q\leq 2\leq p\leq\infty.
\]
\end{theorem}
\begin{proof}
We shall reduce matters to the case of light-cones, so that we can use the same construction as in \S\ref{S_2}. Consider in $V=\SR^n$ a fixed Peirce decomposition, and split the space as the orthogonal direct sum
\Be\label{V1}
V=\cV\oplus\Big[\cV_{\frac12}\oplus\cW\Big],
\Ee
where 
\[ 
\cV:=V_{11}\oplus V_{12}\oplus V_{22}\mand \cW:=\bigoplus_{3\leq i\leq j\leq r} V_{ij}
\]
are, respectively, the subalgebras of upper $2\times 2$ and lower $(r-2)\times (r-2)$ matrices, and where
\[
\cV_{1/2}:=\bigoplus_{1\leq i\leq 2<j\leq r} V_{ij}.
\]
Note that, if $\be'$ denotes the identity element in $\cW$, then we have 
\Be\label{VO}
\O\cap \big[\cV+\be'\big]=\tO+\be',
\Ee
where $\tO$ is the cone (of rank 2) associated with $\cV$.
This last identity is easily proved using the characterization of cones by principal minors in \eqref{minor}. Finally, we select suitable  coordinates in $\cV\equiv \SR^k$ so that $\tO$ is represented as a light-cone, and therefore we can use the construction given in \S\ref{S_2}.

Namely, arguing as in \S\ref{S_d}, we write $\SR^k=\SR^3\times \SR^{k-3}$, and fix two (non-null) functions $\phi\in C^\infty_c(\SR^{k-3})$ and $\psi\in C^\infty_c(\SR^{n-k})$ supported in small balls.
Next, we select a collection of sets $F_j\subset\SR^3$ constructed as in \S\ref{S_21}, and define the functions
\[
f_j=\bbone_{F_j}\otimes\phi\otimes \psi.
\]
We define the vectors \[\bc_j:=(-\tbn_j, 0, 0)\in V
\]
where $\tbn_j=(-1,\bu_j, 0, \ldots,0)\in\SR^k$ are as in \eqref{nj}, that is generated by the longest side directions $\bu_j\in\SR^2$ of the Besicovitch rectangles $R_j$ conforming the sets $F_j$. 

The vectors $\bc_j$ 
are primitive idempotents in $V$ (since so are the vectors $-\tbn_j$ in $\cV$). So, if we define the half-spaces $\Pi_j$ in \eqref{Pij2}, and the corresponding multiplier operators $H_j$ in \eqref{Hj2}, we can make use of Lemma \ref{L_meyer2}. Finally, by orthogonality, it is easily seen that this construction gives
\[
H_j(f_j)=H_j(\bbone_{F_j}\otimes \phi\otimes\psi)=H^{(3)}_j(\bbone_{F_j})\otimes \phi\otimes\psi,
\]
where $H^{(3)}_j$ is the 3-dimensional multiplier associated with half-spaces 
\[\Pi_j^{(3)}=\{\xi\in\SR^3\mid \lan\xi{(-1,\bu_j)}<0\}.
\]
Thus,  Lemma \ref{L_meyer2} implies a 3-dimensional inequality
\[
\|\phi\otimes\psi\|_1\,\Big\|\big(\sum_j|H_j^{(3)}(\bbone_{F_j})|^2\big)^\frac12\Big\|_{L^1(\SB^{(3)})}\,\leq\, c_p\,M\,\|\phi\otimes\psi\|_p\,\Big\|\big(\sum_j|\bbone_{F_j}|^2\big)^\frac12\Big\|_{L^p(\SB^{(3)})},
\]
which we know it cannot hold from \S\ref{S_21}. This completes the proof of Theorem \ref{th1b}.
\end{proof}


\section{Boundedness of Cauchy-Szeg\"o projections}
\label{S_4}
\setcounter{equation}{0}\setcounter{footnote}{0}
\setcounter{figure}{0}

In this section we describe in more detail the question posed at the end of \S1, and derive Corollary \ref{cor1} from Theorem \ref{th1b} and the transfer principle proved in \cite{BB}.

We begin by recalling some basic terminology, referring to the survey paper \cite{vagi} for more detailed definitions and further literature.

\subsection{Cauchy-Szeg\"o projection in a bounded symmetric domain}\label{S_D}

Let $D$ be a 
bounded symmetric domain in $\SC^n$ in a standard realization\footnote{In particular, as in \cite{Bo60}, we assume $0\in D$ and $D$ convex and circular, ie $e^{i\theta} z\in D$ for all $\theta\in\SR$, $z\in D$.}, and let $\Sigma$ denote its Shilov boundary. We write $d\sigma$ for the (normalized) measure in $\Sigma$ which is invariant under the subgroup of conformal automorphisms of $D$ that preserve the origin. The Hardy space $H^p(D)$, $1\leq p<\infty$, is the set of 
holomorphic functions $g:D\to \SC$ such that
\[
\|g\|_{H^p}:=\sup_{0<r<1}\Big[\int_\Sigma |g(rw)|^p\,d\sigma(w)\Big]^\frac1p<\infty.
\]
Every $g\in H^p(D)$ has a boundary value
\[
g_0(w):=\lim_{r\nearrow 1} g(rw), \quad w\in\Sigma,
\]
with convergence in $L^p(\Sigma)$ (and a.e.), and moreover $\|g\|_{H^p}=\|g_0\|_{L^p(\Sigma)}$; see  \cite{Bo60} or \cite[p. 278]{vagi}. In particular, $H^2(D)$ is a Hilbert space, and there is an orthogonal projection $S_D: L^2(\Sigma)\to H^2(D)$. 
Since point evaluations are continuous linear functionals in $H^2$, this operator can be expressed as
\[
S_Dg(z)=\int_\Sigma S_D(z,w)g(w)\,d\sigma(w), \quad z\in D,
\]
where $S_D(z,w)$ is called the \emph{Cauchy-Szeg\"o kernel} of $D$. Passing to boundary values, we consider the operator, defined at least for $g\in L^2(\Sigma)$, 
\Be\label{SSD}
\SS_Dg(\xi):=\lim_{r\nearrow 1}\int_\Sigma S_D(r\xi,w)g(w)\,d\sigma(w), \quad \xi\in \Sigma,
\Ee
which we shall call the \emph{Cauchy-Szeg\"o projection} of $D$. 

\

{\bf Question 1:} \emph{For what values $1\leq p, q \leq \infty$ there exists a continuous extension of $\SS_D:L^p(\Sigma)\to L^q(\Sigma)$? }

\

When $D$ is the unit ball of $\SC^n$, a classical theorem of Kor\'anyi and V\'agi establishes that $\SS_D$ maps $L^p\to L^p$ for all $1<p<\infty$, see \cite{KV71} or \cite[Ch 6]{rudin2}; hence the above question has a positive answer of all $1\leq q\leq p\leq \infty$, except at the endpoints $p=q=1$ and $p=q=\infty$.

However, for (irreducible) domains $D$ of higher rank $\geq2$ the situation is quite different. If $D$ is a domain of \emph{tube-type}, that is, conformally equivalent to a tube domain $\Tu=\SR^n+i\O$ over a symmetric cone $\O\subset\SR^n$, then B\'ekoll\'e and Bonami  showed in \cite[Thm 6]{BB} that $\SS_D$ maps $L^p\to L^p$ iff $p=2$, and left open the study of $(p,q)$ inequalities. Moreover, if $D$ is conformally equivalent to a Siegel domain of type II (ie., not of tube type), then the question is completely open, even regarding $L^p\to L^p$ boundedness (for $p\not=2$).

We remark that, although the kernels $S_D(z,w)$ have geometrically explicit formulas, see e.g. \cite{FK90}, these are not so easily manageable for 
handling the operator $\SS_D$. A common strategy is to transfer the problem, via conformal mappings, to an unbounded realization of $D$, which makes possible to use Fourier transform techniques. We present in the next subsection the needed terminology from the unbounded setting; see \cite[Chapter 3]{sw} or \cite[Ch IX.4]{FK}
for a more detailed exposition.

\subsection{Cauchy-Szeg\"o projection in tube domains}
Let 
\[
\Tu:=\SR^n+i\O\subset\SC^n
\]
denote the tube domain over a symmetric cone $\O$ in $\SR^n$. 
The Hardy space 
$H^p(\Tu)$, $1\leq p<\infty$,  is now defined as the set of holomorphic functions $F$ in $\Tu$ with
\[
\|F\|_{H^p(\Tu)}:=\sup_{y\in\O}\Big[\int_{\SR^n}|F(x+iy)|^p\,dx\Big]^\frac1p\,<\,\infty.
\]
Every $F\in H^p(\Tu)$ has a boundary value $F_0\in L^p(\SR^n)$
\[
F_0(x):=\lim_{y\to0} F(x+iy),\quad x\in\SR^n,
\]
where the limit exists in the $L^p(\SR^n)$-norm
(unrestrictedly in $y\in\O$), and also pointwise at a.e. $x\in\SR^n$ 
if $y$ is restricted to a proper closed subcone $\overline{\O}_1\subset \O\cup\{0\}$; see e.g. \cite[Chapter 3, Theorems 5.5\footnote{The pointwise restricted convergence at a.e. $x$ also holds \emph{non-tangentially}, as described in \cite[p. 119]{sw}. This fact is used later in the transference principle, see also \cite[\S5.2]{BB}.} and 5.6]{sw}. As before, it holds $\|F\|_{H^p}=\|F_0\|_{L^p(\SR^n)}$, and in particular, $H^2(\Tu)$ can be regarded as a closed subspace of $L^2(\SR^n)$. The orthogonal projection $S_{\Tu}:L^2(\SR^n)\to H^2(\Tu)$ can be written as
\[
S_{\Tu}f(z)=\int_{\SR^n} S_\Tu(z,u)f(u)\,du, \quad z\in \Tu,
\]
where the \emph{Cauchy-Szeg\"o kernel} has now the explicit expression
\Be\label{SSTu}
S_\Tu(z,u)=\int_\O e^{2\pi i \lan{z-u}{\xi}}\,d\xi, \quad z\in\Tu,\;u\in\SR^n;
\Ee
see \cite[Theorem IX.4.3]{FK}. Passing to boundary values, we define the \emph{Cauchy-Szeg\"o projection} in $\Tu$ as the operator 
\Be\label{Mp}
\SS_\Tu f(x):=\lim_{y\to 0}S_\Tu f(x+iy)
=\int_\O e^{2\pi i \lan{x}{\xi}}\,\hf(\xi)\,d\xi,
\Ee
at least for $f\in L^2(\SR^n)$. Thus, $\SS_\Tu$  coincides with the Fourier multiplier operator in $\SR^n$ given by the symbol $\bbone_\O$. 
In particular, if $M_p(\SR^n)$ denotes the class of bounded Fourier multipliers in $L^p(\SR^n)$, we have
\[
\SS_\Tu:L^p(\SR^n)\to L^p(\SR^n) \quad \iff \quad\bbone_\O\in M_p(\SR^n).
\]

When $n=1$, 
it goes back to M. Riesz \cite{riesz} that 
the Cauchy projection
maps $L^p(\SR)\to L^p(\SR)$ iff $1<p<\infty$. 
When $n>1$ (and $\O$ is irreducible), the situation changes and only the trivial $L^2$ boundedness of $\SS_\Tu$ can hold. This fact was an open question until the appearance of Fefferman's example; see   \cite[Problem 3]{stein70}.
The precise result is stated below, quoted from \cite[Theorem 5]{BB}.

\begin{theorem}\label{th3}
Let $1\leq p\leq \infty$, and let $\O$ be an irreducible symmetric cone in $\SR^n$ with rank $r\geq2$. Then  $\SS_\Tu$ admits a bounded extension to $L^p(\SR^n)$ only if $p=2$.
\end{theorem}
\begin{proof}
When $r=2$, the result is clear since $\O$ is equivalent to a light-cone.
We sketch the argument for $r>2$, which is not explicit in \cite{BB}. 

Choosing in $V=\SR^n$ the same coordinates as in \eqref{V1}, one obtains an
orthogonal decomposition $\SR^n=\SR^k\oplus \SR^{n-k}$ and a vector $\be_0=(0,\be')\in\SR^{n-k}$ so that
\Be\label{tO}
\O\cap \big[\SR^k\times\{\be_0\}\big]=\tO\times\{\be_0\},
\Ee
where $\tO$ is a cone of rank 2 in $\SR^k$. So, if $\bbone_\O$ belongs to $M_p(\SR^n)$ then, by de Leeuw's theorem, $\bbone_\tO$ belongs to $M_p(\SR^k)$,
which can only happen if $p=2$.
\end{proof}

We finally remark that, in the unbounded setting of $\Tu$, the (global) continuity of $\SS_\Tu:L^p(\SR^n)\to L^q(\SR^n)$ cannot occur when $p\not=q$. However, the question of local $(p,q)$-boundedness can be asked, and as shown in \cite{BB}, it turns out to be equivalent to the Question 1 above. This is the content of the next subsection.

\subsection{The transfer principle}
 Let $D$ be an irreducible bounded symmetric domain of tube type in $\SC^n$, that is, so that there is a conformal bijection
\Be\label{conf1}
\Phi:D\to \Tu,
\Ee
 where  $\O$ is an (irreducible) symmetric cone in $\SR^n$. When $D$ is given in its standard realization, the mapping $\Phi$ is explicit 
  \Be\label{cay}
\Phi(w)=i(\be+w)(\be-w)^{-1}
 \Ee
 and is usually called the \emph{Cayley transform}\footnote{In  \eqref{cay}, the identity element $\be$ and the product and inverse  
are related to the Jordan algebra structure of the underlying (real) space $V=\SR^n$; see \S\ref{S_3} above, or \cite{FK}. 
 }; see e.g. \cite[\S X.2]{FK}.

 The Cauchy-Szeg\"o kernels in $D$ and $\Tu$ are related by the formula
 \Be\label{SDTu}
 S_D(w,w')=\,c_0\, S_{\Tu}(\Phi(w),\Phi(w'))\,J_{\Phi}(w)^\frac12\,\bar{J}_\Phi(w')^\frac12,
 \Ee
 where $J_\Phi$ is the complex jacobian of $\Phi$, and $c_0\in\SC$ is a constant.
This formula can be obtained from a similar identity relating the Bergman kernels of $D$ and $\Tu$, which appears when squaring both sides in \eqref{SDTu}; see \cite[Chapter XIII]{FK}. Also the change of variables
$z=\Phi(w)$, which preserves the Shilov boundaries, maps the measure $d\sigma(w)$ in $\Sigma$ into a measure in $\SR^n$ explicitly given by
\Be
\frac{dx}{|J_\Phi(\Phi^{-1}(x))|}\,=\, \frac{c_1\,dx}{\Dt(e+x^2)^\frac nr};
\label{CV}
\Ee
see \cite[Prop X.2.4]{FK}. We shall not need the explicit formulas, only the property that, for each compact set $K\Subset\Dom\Phi$ it holds
\Be\label{cK}
c_K\leq |J_\Phi(w)|\leq C_K, \quad w\in K,
\Ee
for some constants $C_K\geq c_K>0$. Here, $\Dom\Phi=\{w\in\SC^n\mid\Dt(\be-w)\not=0\}$, and we shall consider   
\[
\Sigma_0:=\Sigma \cap \Dom\Phi=\Phi^{-1}(\SR^n),
\] 
which is an open dense set with full measure in $\Sigma$; see more details in \cite[Prop X.2.3]{FK}.

Below, in analogy to \eqref{SB}, if $\cT$ is an operator and $K\subset \SC^n$, we use the notation $\cT^K$ for the local operator
\[
\cT^Kf=\bbone_{K}\cT\big(f\bbone_{K}\big).
\]
The transference principle of B\'ekoll\'e and Bonami can now be stated as follows, in a slightly more general setting than \cite{BB}. 

\begin{theorem}
\label{th_transf}
Let $D$ and $\Tu$ be as above. 
Then, for every $1\leq p, q\leq \infty$, the following assertions are equivalent
\Benu
\item $\SS_D:L^p(\Sigma)\to L^q(\Sigma)$ is continuous 
\item $\SS^K_D:L^p(\Sigma)\to L^q(\Sigma)$ is continuous, for all compact sets $K\Subset \Sigma_0$ 
\item $\SS^\tK_\Tu:L^p(\SR^n)\to L^q(\SR^n)$ is continous for all compact sets $\tK\subset \SR^n$.
\Eenu
\end{theorem}
\begin{proof}

For completeness, we give some details of  (2) $\Rightarrow$ (3), which is the only case relevant for purposes
 (the implication (1) $\Rightarrow$ (2) being trivial)\footnote{The argument for (2)$\Rightarrow$(1) can be seen
 in \cite[\S3]{BB}; see also \cite[Prop 2.8]{BGN14} for a slightly more detailed presentation. The implication (3)$\Rightarrow$(2), 
which is not needed here, has an entirely similar proof to (2)$\Rightarrow$(3).}.  

Assuming (2), we fix two compact sets  $K\subset\Sigma_0$ and $\tK\subset\SR^n$ such that $\tK=\Phi(K)$. We must show that, for some constant $C>0$ (possibly depending on $K$), it holds
\Be\label{IK}
I:=\int_\tK\big|\SS_\Tu f(x)\big|\,\big|g(x)\big|\,dx\leq C,
\Ee
for every $f$, $g\in C^\infty_c(\tK)$ with $\|f\|_p=\|g\|_{q'}=1$. 
Changing variables $x=\Phi(w)$ in \eqref{IK}, and using \eqref{CV} and \eqref{cK} we have
\Be\label{aux1}
I \leq C_K\,\int_K\big|\SS_\Tu f\big(\Phi(w)\big)\big|\,\big|g\big(\Phi(w)\big)\big|\,d\sigma(w).
\Ee
Now, for a.e. $w_0=\Phi^{-1}(x_0)$ in $\Sigma_0$ we claim that
\Be\label{SSlim}
\SS_\Tu f\big(\Phi(w_0)\big)=\lim_{r\to 1^-} S_\Tu f\big(\Phi(rw_0)\big).
\Ee
Indeed, the existence of the limit follows  from the fact that the radial segment $r\mapsto rw_0$ is transformed into the curve $r\mapsto \Phi(rw_0)$,
which for $r\nearrow 1$ lies in the region of restricted non-tangential convergence of the point $x_0$, due to the angle preserving property of the conformal map $\Phi$. 

If we now fix $w:=rw_0$ (for a given $r<1$), we can write
\Beas 
S_\Tu f\big(\Phi(w)\big) & = & \int_{\SR^n}
S_\Tu\big(\Phi(w), x'\big)\,f(x')\,dx'\\
\mbox{{\tiny ($x'=\Phi(w')$})} & = & 
\int_{\Sigma_0} S_\Tu\big(\Phi(w), \Phi(w')\big)\,f(\Phi(w'))\,|J_\Phi(w')|\,d\sigma(w')\\
& = & J_\Phi(w)^{-1/2}\,\int_{\Sigma_0} S_D(w, w')\,F(w')\,d\sigma(w'),
\Eeas
where we have used \eqref{SDTu} and  have denoted 
\[
F(w'):= c_0^{-1}\,\bar{J}_\Phi(w')^{-1/2}\,f(\Phi(w'))\,|J_\Phi(w')|.
\]
So, combining the previous equalities with \eqref{SSlim} we obtain
\[
\SS_\Tu f\big(\Phi(w_0)\big)= 
J_\Phi(w_0)^{-1/2}\,\SS_D F(w_0), \quad \sae w_0\in\Sigma_0.
\]
Continuing with \eqref{aux1}, we see that
\Beas 
I & \leq & C_K\,c_K^{-1/2}\,\int_K \big|\SS_D F\big|\,\big|g\circ\Phi\big|\,d\sigma\\
& \leq & C_K\,c_K^{-1/2}\,\big\|\SS_D F\big\|_{L^q(K)}\,
\|g\circ \Phi\|_{L^{q'}(\Sigma_0)}\\
& \leq & C'_K\,\|F\|_{L^p(\Sigma_0)}\,
\|g\circ \Phi\|_{L^{q'}(\Sigma_0)},
\Eeas
the last inequality due to the assumption (2). Finally, undoing the change of variables, and using once again \eqref{CV} and \eqref{cK}, we deduce that
\[
\|F\|_{L^p(\Sigma_0)}\,\leq\, c'_K\,\|f\|_{L^p(\SR^n)}=c'_K,
\]
and likewise $\|g\circ \Phi\|_{L^{q'}(\Sigma_0)}\leq c''_K$. This establishes \eqref{IK}.
\end{proof}

\subsection{Proof of Corollary \ref{cor1}}
This is now a direct consequence of \eqref{Mp}, Theorem \ref{th1b}, and the implications
(1) $\Rightarrow$  (2)  $\Rightarrow$  (3) in Theorem \ref{th_transf}.
\ProofEnd

\subsection{An application: duality of $H^p$ spaces}
\label{S_dual}

Let $D$ be a bounded symmetric domain in $\SC^n$, and let $H^p$, $1\leq p<\infty$, be the associated Hardy space as in \S\ref{S_D}.
It is well-known that, with the usual duality pairing, the (isomorphic) identity 
\[
(H^p)^*=H^{p'}
\]
holds under the assumption that $\SS_D:L^p\to L^p$ is bounded (where as usual, $\frac1p+\frac1{p'}=1$). 
In particular, if $D$ is the unit ball of $\SC^n$, the identity holds for all $1<p<\infty$. 
However, in (irreducible) bounded symmetric domains of higher rank, where the boundedness of $\SS_D$ fails (or is unknown), 
the structure of the dual space $(H^p)^*$ is still quite open; see e.g. the question posed in \cite[\S5.1]{zhu3}.

An abstract argument, based in the Hahn-Banach theorem, gives the following known identification.

\begin{lemma}\label{L_dual}
Let $1\leq p\leq 2$. Then, with the usual duality pairing, 
the following isometrical identity holds
\Be
\label{Hp_dual}
(H^p)^*=S_D(L^{p'}),
\Ee
where the right hand side is the space of all $G=S_D(g)$, for some $g\in L^{p'}(\Sigma)$, endowed with the norm
\[
\|G\|_{S_D(L^{p'})}:=\inf_{g\in L^{p'}\colon S_D(g)=G}\|g\|_{L^{p'}}.
\]
\end{lemma}
\begin{proof} We sketch the standard proof for completeness. 
First note that we may regard $S_D(L^{p'})$ as a subspace of $H^2$, since $1\leq p\leq 2$. 
Next, every $G=S_D(g)$, with $g\in L^{p'}(\Sigma)$, defines an element $\Phi_G$ in $(H^p)^*$ by the standard duality pairing,
namely
\[
\Phi_G(F):=\lan FG=\lan{f_0}{\SS_Dg}=\int_\Sigma f_0\,\bar{g}\,d\sigma,\quad F\in H^p\cap H^2,
\] 
where $f_0$ is the boundary limit of $F$. Moreover, 
\[
\big|\Phi_G(F)\big|\leq \|f_0\|_{L^p}\|g\|_{L^{p'}}=\|F\|_{H^p}\|g\|_{L^{p'}},
\]
so taking the infimum over all $g$ with $S_D(g)=G$ one obtains 
\[
\|\Phi_G\|_{(H^p)^*}\leq \|G\|_{S_D(L^{p'})}.
\] 
The correspondence $G\mapsto \Phi_G$ is injective, since testing with $K_z:=S_D(\cdot, z)\in H^2\cap H^p$, for each fixed $z\in D$, gives
\[
\Phi_G(K_z)=\overline{\int_\Sigma S_D(z,w)g(w)\,d\sigma(w)}=\overline{G(z)};
\]
so, if $\Phi_G=0$ then $G=0$. 

Finally, to see surjectivity, let $\Phi\in (H^p)^*$. Passing to boundary values we may regard $H^p(D)$ as a closed subspace of $L^p(\Sigma)$.
So, by the Hahn-Banach theorem $\Phi$ has a continuous extension $\tPhi$ to $L^p(\Sigma)$ with 
\[
\|\tPhi\|_{(L^p)^*}\leq \|\Phi\|_{(H^p)^*}.
\]
Thus, we can find $g\in L^{p'}(\Sigma)$ with $\|g\|_{L^{p'}}=\|\tPhi\|_{(L^p)^*}$ and 
\[
\tPhi(f)=\int_\Sigma f\,\bar{g}\,d\sigma, \quad f\in L^p.
\]
Then, letting $G:=S_D(g)$, and using the above notation one has
\[
\Phi_G(F)=\lan FG =\int_\Sigma f_0\,\bar{g}\,d\sigma=  \tPhi(f_0)=\Phi(F),\quad F\in H^p\cap H^2, 
\]
which implies that $\Phi_G=\Phi$. Moreover, 
\[
\|G\|_{S_D(L^{p'})}\leq \|g\|_{L^{p'}}=\|\tPhi\|_{(L^p)^*}\leq \|\Phi\|_{(H^p)^*}.
\]
This shows that the pairing $G\mapsto \Phi_G$ is an (anti)-linear isometric isomorphism from $S_D(L^{p'})$ into $(H^p)^*$.
\end{proof}

As a consequence of Lemma \ref{L_dual} and Corollary \ref{cor1} we obtain the following.

\begin{corollary}\label{th_dual}
Let $D\subset \SC^n$ be an irreducible bounded symmetric domain of tube type with rank $\geq2$. If $1\leq p\leq 2$ then (for the canonical inclusion) we have 
\Be\label{Hd}
(H^p)^* \not\hookrightarrow H^q, \quad \forall q>2.
\Ee
\end{corollary}
\begin{proof}

In view of Lemma \ref{L_dual}, we may regard $(H^p)^*$ as a space of 
holomorphic functions in $D$, whose boundary values exist and belong at least to $L^2(\Sigma)$.
Now, if the inclusion in \eqref{Hd} was true this would imply, for some constant $C>0$, that 
\[
\|\SS_D g\|_{L^{q}}=\|S_Dg\|_{H^{q}}\leq\, C\,\|S_Dg\|_{(H^p)^*}\leq \,C\,\|g\|_{L^{p'}}\,,
\]
using in the last step Lemma \ref{L_dual}. 
Thus $\SS_D$ maps $L^{p'}\to L^q$ with norm bounded by $C$, which contradicts Corollary \ref{cor1} if $q>2$.
\end{proof}
\BR
For $p=1$, if one attempts to define $BMOA:=(H^1)^*$, in analogy to one of the various characterizations in the unit ball setting \cite[Thm 5.1]{zhu2}, then the boundary values of functions in $BMOA$
will in general not belong to $L^q(\Sigma)$, for any $q>2$. 
In particular, no embeddings of John-Nirenberg type will hold with such definition (compare with \cite[Ex 5.19]{zhu2} for the unit ball of $\SC^n$). A similar result for the dual of the Hardy space in the infinite torus $(H^1(\ST^\infty))^*$, has recently been obtained by Konyagin et al., see \cite[Corollary 2.3]{KQSS22}.
\ER

\

\emph{Acknowledgements}: The authors thank an anonymous referee for the careful reading and useful comments.

\bibliographystyle{plain}

\end{document}